\documentclass[11pt]{article}

\usepackage{amsmath,amsthm,amsfonts,mathrsfs,mathtools,bbm,psfrag}
\usepackage{a4}
\usepackage{fullpage}
\usepackage{hyperref}
\usepackage{graphicx,graphics,epsfig,color}

\newcommand{\R}{\mathbb{R}}

\newcommand{\N}{\mathbb{N}}

\newcommand{\cG}{\mathcal{G}}

\newcommand{\Pmic}{\mathrm{P}_{\mathrm{mic}}}
\newcommand{\Pcan}{\mathrm{P}_{\mathrm{can}}}
\newcommand{\Emic}{\mathrm{E}_{\mathrm{mic}}}
\newcommand{\Ecan}{\mathrm{E}_{\mathrm{can}}}

\newcommand{\eee}{\mathrm{e}}

\newcommand{\Wt}{\tilde{\mathcal{W}}}
\newcommand{\des}{\delta_\square}

\newcommand{\h}{\tilde{h}}
\newcommand{\W}{\mathcal{W}}

\newcommand{\limn}{\lim_{n\to\infty}}

\newcommand{\that}{\hat{\theta}}

\newcommand{\dd}{\text{d}}
\newcommand{\dx}{\,\mathrm{d}x}
\newcommand{\dy}{\,\mathrm{d}y}

\DeclareMathOperator{\op}{op}

 \makeatletter
 \@addtoreset{equation}{section}
 \makeatother

 \makeatletter
 \@addtoreset{enunciato}{section}
 \makeatother

 \newcounter{enunciato}[section]

 \newtheorem{ittheorem}{Theorem}
 \newtheorem{itlemma}{Lemma}
 \newtheorem{itproposition}{Proposition}
 \newtheorem{itdefinition}{Definition}
 \theoremstyle{definition}
 \newtheorem{itremark}{Remark}
 \newtheorem{itclaim}{Claim}
 \newtheorem{itfact}{Fact}
 \newtheorem{itexample}{Example}
 \newtheorem{itconjecture}{Conjecture}
 \newtheorem{itobservation}{Observation} 
 \newtheorem{itcorollary}{Corollary} 
 \newtheorem{itquestion}{Question} 
 \newtheorem{itworkinghypothesis}{Working hypothesis}

 \newenvironment{theorem}{\addtocounter{enunciato}{1}
 \begin{ittheorem}}{\end{ittheorem}}

 \newenvironment{lemma}{\addtocounter{enunciato}{1}
 \begin{itlemma}}{\end{itlemma}}

 \newenvironment{proposition}{\addtocounter{enunciato}{1}
 \begin{itproposition}}{\end{itproposition}}

 \newenvironment{definition}{\addtocounter{enunciato}{1}
 \begin{itdefinition}}{\end{itdefinition}}

 \newenvironment{remark}{\addtocounter{enunciato}{1}
 \begin{itremark}}{\end{itremark}}

 \newenvironment{claim}{\addtocounter{enunciato}{1}
 \begin{itclaim}}{\end{itclaim}}

 \newenvironment{fact}{\addtocounter{enunciato}{1}
 \begin{itfact}}{\end{itfact}}

 \newenvironment{conjecture}{\addtocounter{enunciato}{1}
 \begin{itconjecture}}{\end{itconjecture}}

 \newenvironment{corollary}{\addtocounter{enunciato}{1}
 \begin{itcorollary}}{\end{itcorollary}}

 \newcommand{\be}[1]{\begin{equation}\label{#1}}
 \newcommand{\ee}{\end{equation}}

 \newcommand{\bl}[1]{\begin{lemma}\label{#1}}
 \newcommand{\el}{\end{lemma}}

 \newcommand{\bt}[1]{\begin{theorem}\label{#1}}
 \newcommand{\et}{\end{theorem}}

 \newcommand{\bd}[1]{\begin{definition}\label{#1}}
 \newcommand{\ed}{\end{definition}}

 \newcommand{\bcl}[1]{\begin{claim}\label{#1}}
 \newcommand{\ecl}{\end{claim}}

 \newcommand{\bfact}[1]{\begin{fact}\label{#1}}
 \newcommand{\efact}{\end{fact}}

 \newcommand{\bp}[1]{\begin{proposition}\label{#1}}
 \newcommand{\ep}{\end{proposition}}

 \newcommand{\bc}[1]{\begin{corollary}\label{#1}}
 \newcommand{\ec}{\end{corollary}}

 \newcommand{\bcj}[1]{\begin{conjecture}\label{#1}}
 \newcommand{\ecj}{\end{conjecture}}

 \newcommand{\bpr}{\begin{proof}}
 \newcommand{\epr}{\end{proof}}

 \newcommand{\bprl}[1]{\begin{proofof}{\it\ref{#1}}.\,\,}
 \newcommand{\eprl}{\end{proofof}}

 \newcommand{\bi}{\begin{itemize}}
 \newcommand{\ei}{\end{itemize}}

 \newcommand{\ben}{\begin{enumerate}}
 \newcommand{\een}{\end{enumerate}}

 \allowdisplaybreaks

\begin{document}

\title{Breaking of ensemble equivalence for\\ 
dense random graphs under a single constraint}

\date{\today}

\author{
\renewcommand{\thefootnote}{\arabic{footnote}}
F.\ den Hollander
\footnotemark[1]
\\
\renewcommand{\thefootnote}{\arabic{footnote}}
M.\ Markering
\footnotemark[2]
}

\footnotetext[1]{
Mathematical Institute, Leiden University, P.O.\ Box 9512,
2300 RA Leiden, The Netherlands,\\
{\sl denholla@math.leidenuniv.nl}
}

\footnotetext[2]{
Mathematical Institute, Leiden University, P.O.\ Box 9512,
2300 RA Leiden, The Netherlands,\\
{\sl maartenmarkering@outlook.com}
}

\maketitle

\begin{abstract}
Two ensembles are frequently used to model random graphs subject to constraints: the \emph{microcanonical ensemble} (= hard constraint) and the \emph{canonical ensemble} (= soft constraint). It is said that \emph{breaking of ensemble equivalence} (BEE) occurs when the specific relative entropy of the two ensembles does not vanish as the size of the graph tends to infinity. The latter means that it matters for the scaling properties of the graph whether the constraint is met for every single realisation of the graph or only holds as an ensemble average. Various examples were analysed in the literature, and the specific relative entropy was computed as a function of the constraint. It was found that BEE is the rule rather than the exception for two classes: \emph{sparse} random graphs when the \emph{number} of constraints is of the order of the number of vertices and \emph{dense} random graphs when there are \emph{two or more} constraints that are \emph{frustrated}. 

In the present paper we establish BEE for a third class: dense random graphs with a \emph{single} constraint, namely, on the density of a given finite simple graph. In doing so we solve the open problem as to whether BEE is possible under a single constraint. We show that BEE occurs only in a certain range of choices for the density and the number of edges of the simple graph, which we refer to as the BEE-phase. We show that, in part of the BEE-phase, there is a gap between the scaling limits of the averages of the maximal eigenvalue of the adjacency matrix of the random graph under the two ensembles, a property that is referred to as \emph{spectral signature} of BEE. Proofs are based on an analysis of the variational formula on the space of graphons for the limiting specific relative entropy derived in \cite{dHMRS18}, in combination with an identification of the minimising graphons and replica symmetry arguments. We show that in the replica symmetric region of the BEE-phase, as the size of the graph tends to infinity, the microcanonical ensemble behaves like an Erd\H{o}s-R\'enyi random graph, while the canonical ensemble behaves like a \emph{mixture} of two Erd\H{o}s-R\'enyi random graphs. In other words, BEE is due to \emph{coexistence} of two densities.

\vspace{0.5cm}

\emph{Acknowledgement:} The research in this paper was supported through NWO Gravitation Grant NETWORKS 024.002.003. 
\end{abstract}

\newpage


\section{Introduction and main results}
\label{s.intro}

Section~\ref{ss.back} provides the background and the motivation behind our paper. Section~\ref{ss.prelim} states the definition of the microcanonical and the canonical ensemble in the context of constrained random graphs, recalls the notion of ensemble equivalence, lists the key definitions of graphons and subgraph counts, and gives the variational characterisation of the specific relative entropy of the two ensembles for dense random graphs derived in \cite{dHMRS18}, which is the main tool in our paper. Section~\ref{ss.main} states our main theorems. Section~\ref{ss.typgraph} identifies the typical graphs under the two ensembles.  Section~\ref{ss.disc} offers a brief discussion and an outline of the remainder of the paper.    


\subsection{Background and motivation}
\label{ss.back}

In this paper we analyse random graphs that are subject to \emph{constraints}. Statistical physics prescribes which probability distribution on the set of graphs we should choose when we want to model a given type of constraint \cite{G02}. Two important choices are: (1) the \emph{microcanonical ensemble}, where the constraints are satisfied by each individual graph; (2) the \emph{canonical ensemble}, where the constraints are satisfied as ensemble averages. For random graphs that are large but finite, the two ensembles represent different empirical situations. One of the cornerstones of statistical physics is that the two ensembles become equivalent in the thermodynamic limit, which in our setting corresponds to letting the size of the graph tend to infinity. However, this property does \emph{not} hold in general. We refer to \cite{T15} for more background on the phenomenon of \emph{breaking of ensemble equivalence} (BEE).

BEE has been investigated for various choices of constraints, including the degree sequence and the total number of subgraphs of a specific type. The key distinctive object is the \emph{relative entropy} $S_n$ of the microcanonical ensemble with respect to the canonical ensemble when the graph has $n$ vertices. In the \emph{sparse regime},where the number of edges per vertex remains bounded, the relevant quantity is $s_\infty = \lim_{n\to\infty} n^{-1} S_n$, because $n$ is the scale of the number of \emph{vertices}. In the \emph{dense regime}, where the number of edges per vertex is of the order of the number of vertices, the relevant quantity is $s_\infty = \lim_{n\to\infty} n^{-2} S_n$, because $n^2$ is the scale of the number of \emph{edges}. 
\begin{itemize}
\item
\underline{Sparse regime:}
In \cite{SdMdHG15, GHR17,GHR18} it was shown that constraining the degrees of \emph{all} the vertices leads to BEE, even when the graph consists of multiple communities. An explicit formula was derived for $s_\infty$ in terms of the limit of the empirical degree distribution of the constraint. In \cite{GS18} a formula was put forward that expresses the specific relative entropy in terms of a covariance matrix under the canonical ensemble. 
\item
\underline{Dense regime:}
In \cite{dHMRS18} it was shown that constraining the densities of a finite number of subgraphs may lead to BEE. The analysis relied on the \emph{large deviation principle for graphons} associated with the Erd\H{o}s-R\'enyi (ER) random graph \cite{C15,CV11}. The main result was a variational formula for $s_\infty$ in the space of graphons. In \cite{dHMRS19}, for the special case where the constraint is on the densities of the edges and triangles, it was shown that $s_\infty>0$ when the constraints are \emph{frustrated}, i.e., do not lie on the ER-line where the density of triangles is the third power of the density of edges. Moreover, the asymptotics of $s_\infty$ near the ER-line was identified, and turns out to depend on whether the ER-line is approached from above or below. 
\end{itemize} 

It is an open problem whether a \emph{single} constraint may lead to BEE \cite{dHMRS18}. It was believed that this cannot be the case, because for a single constraint there is no frustration. The goal of the present paper is to show that this intuition is wrong: we condition on the density of a given finite simple graph and prove that BEE occurs in a certain \emph{range} of choices for the density and the number of edges of the simple graph, which we refer to as the BEE-phase. We analyse how $s_\infty$ tends to zero near the curve that borders the BEE-phase. This phase transition is \emph{unlike} any of the phenomena surrounding BEE observed before. In our case, BEE is due to the \emph{coexistence} of two densities in the BEE-phase, similar to the phase transition between water and ice. Thus, our paper provides new insight into the mechanisms causing BEE.

In \cite{DGHM20} the gap $\Delta_n$ between the \emph{averages} of the maximal eigenvalue of the adjacency matrix of a constrained random graph under the two ensembles was considered. A \emph{working hypothesis} was put forward, stating that BEE is equivalent to this gap not vanishing in the limit as $n\to\infty$. For a random regular graph with a fixed degree, this equivalence was proved for a range of degrees that interpolates between the sparse and the dense regime. In the present paper we prove the same for the single constraint. In particular, we compute $\delta_\infty = \lim_{n\to\infty} n^{-1} \Delta_n$, show that $\delta_\infty \neq 0$ if and only if the density and the number of edges of the simple graph fall in the BEE-phase, and analyse how $\delta_\infty$ tends to zero near the curve that borders the BEE-phase. 

We will see that the notions of \emph{replica symmetry} and \emph{replica symmetry breaking} highlighted in \cite{LZ16} play an important role. In the regime of replica symmetry we have a complete identification of $s_\infty$ and $\delta_\infty$, in the regime of replica symmetry breaking some pieces of the characterisation are missing. Furthermore, we establish a direct connection between the region of replica symmetry for regular graphs and the region of ensemble equivalence.
  

\subsection{Definitions and preliminaries}
\label{ss.prelim}

In this section, which is partly lifted from \cite{dHMRS18}, we present the definitions of the main concepts to be used in the sequel, together with some key results from prior work. We consider scalar-valued constraints, even though \cite{dHMRS18} deals with more general vector-valued constraints. 

Section \ref{sss.ensemble} presents the formal definition of the two ensembles and the definition of ensemble equivalence in the dense regime. Section \ref{sss.graphons} recalls some basic facts about graphons. Section \ref{sss.varchar} recalls the variational characterisation of ensemble equivalence proven in \cite{dHMRS18}. Section \ref{sss.maxeig} looks at the average of the maximal eigenvalue value of the adjacency matrix in the two ensembles and recalls a working hypothesis put forward in \cite{DGHM20} that links ensemble equivalence to a vanishing gap between the two averages. 
 

\subsubsection{Microcanonical ensemble, canonical ensemble, relative entropy}
\label{sss.ensemble}

For $n \in \N$, let $\cG_n$ denote the set of all $2^{\binom{n}{2}}$ simple undirected graphs with $n$ vertices.\ Let $T$ denote a scalar-valued function on $\cG_n$, and $T^*$ a specific scalar that is \emph{graphical}, i.e., realisable by at least one graph in $\cG_n$. Given $T^*$, the \emph{microcanonical ensemble} is the probability distribution $\Pmic$ on $\cG_n$ with \emph{hard constraint} $T^*$ defined as
\begin{equation}
\Pmic(G) :=
\left\{
\begin{array}{ll} 
| \{G \in \cG_n\colon\, T(G) = T^* \}|^{-1}, \quad & \text{if } T(G) = T^*, \\ 
0, & \text{otherwise},
\end{array}
\right. \qquad G\in \cG_n.
\label{eq:PM}
\end{equation}
The \emph{canonical ensemble} $\Pcan$ is the unique probability distribution on $\cG_n$ that maximises the \emph{entropy} 
\begin{equation}
S_n({\rm P}) := - \sum_{G \in \cG_n}{\rm P}(G) \log {\rm P}(G)
\end{equation}
subject to the \emph{soft constraint} $\langle T \rangle  := \sum_{G \in \cG_n} T(G)\,{\rm P}(G) = T^*$. This gives the formula \cite{J57}
\begin{equation}
\Pcan(G) := \frac{1}{Z(\theta^*)}\,\eee^{\theta^*T(G)}, \qquad G \in \cG_n,
\label{eq:PC}
\end{equation}
with $Z(\theta^*)$ the \emph{partition function} In \eqref{eq:PC}, the Lagrange multiplier $\theta^*$ must be set to the unique value that realises $\langle T \rangle  = T^*$, see \cite[Equation (2.6)-(2.7)]{dHMRS18}.

The \emph{relative entropy} of $\Pmic$ with respect to  $\Pcan$ is defined as
\begin{equation}
S_n(\Pmic \mid \Pcan) 
:= \sum_{G \in \cG_n} \Pmic(G) \log \frac{\Pmic(G)}{\Pcan(G)}.
\label{eq:KL1}
\end{equation}
For any $G_1,G_2\in\cG_n$, $\Pcan(G_1)=\Pcan(G_2)$ whenever $T(G_1)=T(G_2)$, i.e., the canonical probability is the same for all graphs with the same value of the constraint. We may therefore rewrite \eqref{eq:KL1} as
\begin{equation}
S_n(\Pmic \mid \Pcan) = \log \frac{\Pmic(G^*)}{\Pcan(G^*)},
\label{eq:KL2}
\end{equation}
where $G^*$ is \emph{any} graph in $\cG_n$ such that $T(G^*) =T^*$.

\begin{remark}
\label{rmk.convergence}
Both the constraint $T^*$ and the Lagrange multiplier $\theta^*$ in general depend on $n$, i.e., $T^*=T^*_n$ and $\theta^* = \theta^*_n$. We consider constraints that converge when we pass to the limit $n\to\infty$, i.e., 
\begin{equation}
\label{eq: Assumption T}
\lim_{n\to\infty} T^*_n =: T^*_\infty.
\end{equation}
Consequently, we expect that  
\begin{equation}
\label{eq:Assumption}
\lim_{n\to\infty} \theta^*_n =: \theta^*_\infty.
\end{equation}
Throughout the paper we \emph{assume} that \eqref{eq:Assumption} holds. If convergence fails, then we may still consider subsequential convergence. The subtleties concerning \eqref{eq:Assumption} were discussed in detail in \cite[Appendix A]{dHMRS18}.
\end{remark}

All the quantities above depend on $n$. In order not to burden the notation, we exhibit this $n$-dependence only in the symbols $\cG_n$ and $S_n(\Pmic \mid \Pcan)$. When we pass to the limit $n\to\infty$, we need to specify how $T(G)$, $T^*$ and $\theta^*$ are chosen to depend on $n$. We refer the reader to \cite{dHMRS18}, where this issue was discussed in detail. 

\begin{definition}{\bf [Ensemble equivalence]}
\label{def: sinfty}
{\rm In the dense regime, if 
\begin{equation}
\label{eq: def relative entropy limit}
s_{\infty} : = \lim_{n\to\infty} n^{-2} S_n(\Pmic \mid \Pcan)=0,
\end{equation}
then $\Pmic$ and $\Pcan$ are said to be {\em equivalent}.}
\end{definition}
This particular notion of ensemble equivalence is known as \emph{measure equivalence} of ensembles and is standard in the study of ensemble equivalence of networks. Other notions of ensemble equivalence are \emph{thermodynamic equivalence} and \emph{macrostate equivalence}. Under certain hypotheses, the three notions have been shown to be equivalent for physical systems \cite{T15}. We refer the reader to \cite{T15} and \cite{TET04} for further discussion of different notions of ensemble equivalence.


\subsubsection{Graphons}
\label{sss.graphons}

There is a natural way to embed a simple graph on $n$ vertices in a space of functions called \emph{graphons}. Let $\W$ be the space of functions $h\colon\,[0,1]^2 \to [0,1]$ such that $h(x,y) = h(y,x)$ for all $(x,y) \in [0,1]^2$. A finite simple graph $G$ on $n$ vertices can be represented as a graphon $h^{G} \in \W$ in a natural way as
\begin{equation}
\label{graphondef}
h^{G}(x,y) := \left\{ \begin{array}{ll} 
1 &\mbox{if there is an edge between vertex } \lceil{nx}\rceil \mbox{ and vertex } \lceil{ny}\rceil,\\
0 &\mbox{otherwise}.
\end{array} 
\right.
\end{equation}
The space of graphons $\W$ is endowed with the \emph{cut distance}
\begin{equation}
d_{\square} (h_1,h_2) := \sup_{S,T\subset [0,1]} 
\left|\int_{S\times T} \dd x\,\dd y\,[h_1(x,y) - h_2(x,y)]\right|,
\qquad h_1,h_2 \in \W.
\end{equation}
On $\W$ there is a natural equivalence relation $\sim$. Let $\Sigma$ be the space of measure-preserving bijections $\sigma\colon\, [0,1] \to [0,1]$. Then $h_1(x,y)\sim h_2(x,y)$ if $\delta_{\square}(h_1,h_2)=0$, where $\delta_{\square}$ is the \emph{cut metric} defined by 
\begin{equation}
\label{deltam}
\delta_{\square}(\tilde{h}_1,\tilde{h}_2) 
:= \inf _{\sigma_1,\sigma_2 \in \Sigma} d_{\square}(h_1^{\sigma_1}, h_2^{\sigma_2}),
\qquad \tilde{h}_1,\tilde{h}_2 \in \Wt,
\end{equation}
with $h^\sigma(x,y)=h(\sigma x,\sigma y)$. This equivalence relation yields the quotient space $(\Wt,\delta_{\square})$.
As noted above, we suppress the $n$-dependence. Thus, by $G$ we denote any simple graph on $n$ vertices, by $h^G$ its image in the graphon space $\W$, and by $\tilde{h}^G$ its image in the quotient space $\Wt$. For a more detailed description of the structure of the space $(\Wt,\delta_{\square})$ we refer  to \cite{BCLSV08,BCLSV12,DGKR15}. 

For $h \in \Wt$ and $F$ a finite simple graph with $m$ vertices and edge set $E(F)$, define 
\begin{equation}
\label{eq:homdensity}
t(F,h) := \int_{[0,1]^m} \prod_{\{i,j\} \in E(F)} h(x_i,x_j) \dx_1\ldots\dx_m. 
\end{equation}
Then the \emph{homomorphism density} of $F$ in $G$ equals $t(F,h^G)$, where $h^G$ is the empirical graphon defined in \eqref{graphondef}. 
 
In this paper we focus on the special case where the constraint $T(G)=T(h^G):=t(F,h^G)$ is on the \emph{homomorphism density} $T^*_n$ of a \emph{specific subgraph} $F$. The map $T$ is well-defined on both the space of graphs $\cG_n$ for each $n$ as well as the space of graphons. Rewriting \eqref{eq:PC}, we obtain
\begin{equation}
\label{eq:CPD}
\Pcan(G) = \eee^{n^2\big[\theta_n^* T(G)-\psi_n(\theta_n^*)\big]},
\qquad G \in \cG_n,
\end{equation}
where 
\begin{equation}
\label{eq:PF}
\psi_n(\theta_n^*) := \frac{1}{n^2}\log\sum_{G\in\cG_n} 
\eee^{n^2 [\theta_n^* T(G)]} = \frac{1}{n^2}\log Z(\theta_n^*).
\end{equation}
It turns out that, under the scaling $n^2$, the function $\psi_n$ converges. Hence, rewriting \eqref{eq:PC} in this form aids us in the analysis of the canonical ensemble.
 

\subsubsection{Variational characterisation of ensemble equivalence}
\label{sss.varchar}

In order to characterise the asymptotic behaviour of the two ensembles, the entropy function of a Bernoulli random variable is essential. For $u\in [0,1]$, let 
\begin{equation}
\label{Idef1}
I(u) := \tfrac{1}{2}u\log u +\tfrac{1}{2}(1-u)\log(1-u).
\end{equation}
Extend the domain of this function to the graphon space $\W$ by defining  
\begin{equation}
\label{Idef}
I(h) := \int_{[0,1]^2} \dd x\, \dd y\,\,I(h(x,y))
\end{equation}
(with the convention that $0\log0:=0$). On the quotient space $(\Wt,\delta_{\square})$, define $I(\tilde{h}) = I(h)$, where $h$ is any element of the equivalence class $\tilde{h}$. Note that $I(h)$ takes values in $[-\tfrac12\log 2,0]$. Apart from a shift by $\tfrac12\log 2$, $h \mapsto I(h)$ plays the role of the rate function in the large deviation principle for the empirical graphon associated with the Erd\H{o}s-R\'enyi random graph, derived in \cite{CV11}.

The key result in \cite{dHMRS18} is the following variational formula for $s_\infty$, where
\begin{equation}
\Wt^* := \{\tilde{h}\in \Wt\colon\,T(h) = T^*_\infty\}
\end{equation}
is the subspace of all graphons that meet the constraint $T^*_\infty$. This is a compact set, since $T$ is continuous in the cut metric and $(\Wt,\delta_{\square})$ is compact \cite{LS09}.

\begin{theorem}{\bf [Variational characterisation of ensemble equivalence]}
\label{th:Limit}
Subject to \eqref{eq: Assumption T} and \eqref{eq:Assumption},
\begin{equation}
\label{sinftydef}
\lim_{n\to\infty} n^{-2} S_n(\Pmic \mid \Pcan) =: s_\infty
\end{equation}
with
\begin{equation}
\label{varreprsinfty}
s_\infty = \sup_{\tilde{h}\in \Wt} \left[\theta^*_\infty T(\tilde{h})-I(\tilde{h})\right]
-\sup_{\tilde{h}\in \Wt^*} \left[\theta^*_\infty T(\tilde{h}) - I(\tilde{h})\right].
\end{equation}
\end{theorem}

\noindent
Theorem~\ref{th:Limit} and the compactness of $\Wt^*$ give us a \emph{variational characterisation} of ensemble equivalence: $s_\infty = 0$ if and only if at least one of the maximisers of $\theta^*_\infty T(\tilde{h})-I(\tilde{h})$ in $\Wt$ also lies in $\Wt^* \subset \Wt$, i.e., satisfies the hard constraint.

We need the following lemma, which relates $T_\infty^*$ and $\theta_\infty^*$ without requiring knowledge  of $T_n^*$ and $\theta_n^*$. 

\begin{lemma}
\label{lemma:tuningparameter}
Under the assumptions in \eqref{eq: Assumption T} and \eqref{eq:Assumption},
\begin{equation}
\theta_\infty^* = \arg\max_{\theta \in \R} [\theta T_\infty^*-\psi_\infty(\theta)],
\end{equation}
where
\begin{equation}
\psi_\infty(\theta):=\limn\psi_n(\theta)=\sup_{\h\in\Wt}\left[\theta T(\h)-I(\h)\right].
\end{equation}
\end{lemma}

\begin{proof}
For every $n \in \N$,
\begin{equation}
\theta_n^* 
= \arg\max_{\theta \in \R} \big[n^2 [\theta T_n^*-\psi_n(\theta)]\big]
= \arg\max_{\theta \in \R} [\theta T_n^*-\psi_n(\theta)].
\end{equation}
Let $f_n(\theta,T_n^*) := \theta T^*_n - \psi_n(\theta)$ and $f_\infty(\theta,T_\infty^*) = \theta T^*_\infty-\psi_\infty(\theta)$. By \cite[Theorem 3.2 and Lemma A.1]{dHMRS18},
\begin{equation}
\begin{split}
f_n(\theta_n^*,T_n^*) 
&= \sup_{\theta\in\R} f_n(\theta,T_n^*) = [\theta_n^* T_n^*-\psi_n(\theta_n^*)]
\to [\theta_\infty^*T_\infty^* - \psi_\infty(\theta^*_\infty)]\\
&= \theta_\infty^* T_\infty^* - \sup_{\h\in\Wt}\, \big[\theta_\infty^* T(\h)-I(\h)\big]
= f_\infty(\theta_\infty^*,T_\infty^*), \qquad n \to \infty.
\end{split}
\end{equation}
Furthermore, for every $\theta \in \R$, $f_n(\theta,T_n^*) \leq f_n(\theta_n^*,T_n^*)$, and hence $f_\infty(\theta,T_\infty^*) = \lim_{n\to\infty} f_n(\theta,T_n^*) \leq \lim_{n\to\infty} f_n(\theta_n^*,T_n^*) = f_\infty(\theta_\infty^*,T_\infty^*)$, so that $\theta_\infty^*$ is a maximiser of $f_\infty(\cdot,T_\infty^*)$.
\end{proof}


\subsubsection{Maximal eigenvalue of the adjacency matrix}
\label{sss.maxeig}

In \cite{DGHM20} a \emph{working hypothesis} was put forward, stating that breaking of ensemble equivalence is manifested by a gap between the scaling limits of the averages of the maximal eigenvalue of the adjacency matrix of the random graph under the two ensembles. More precisely, let $\lambda_n$ denote the maximal eigenvalue of the adjacency matrix of $G \in \cG_n$. Then the working hypothesis is that 
\begin{equation}
\begin{array}{llll}
&\lim_{n\to\infty} \Delta_n \neq 0 &\Longrightarrow &\mathrm{BEE},\\ 
&\mathrm{BEE} &\Longrightarrow &\lim_{n\to\infty} \Delta_n \neq 0
\text{ apart from exceptional constraints},
\end{array}
\end{equation}         
with
\begin{equation}
\Delta_n := \Ecan[\lambda_n] - \Emic[\lambda_n].
\end{equation}
In \cite{DGHM20} this equivalence was proven for the specific example where the constraint is on all the degrees being equal to $d(n)$, with $(\log n)^\beta \leq d(n) \leq n - (\log n)^\beta$ for some $\beta \in (6,\infty)$. It turns out that BEE occurs and that $\lim_{n\to\infty} \Delta_n = 1-p$ when $\lim_{n\to\infty} n^{-1} d(n)=p \in [0,1]$, i.e., the exceptional constraints correspond to the ultra-dense regime where $p=1$. 

For our single constraint in the dense regime, we will be interested in the quantity
\begin{equation}
\delta_\infty := \lim_{n\to\infty} n^{-1} \Delta_n.
\end{equation} 


\subsection{Main results}
\label{ss.main}

In what follows, $F$ is any finite simple graph with $k$ edges, and the constraint is on the homomorphism density of $F$ being equal to $T^*_n$, as defined in Section \ref{sss.graphons}. Recall from Remark \ref{rmk.convergence} that we assume that $(T^*_n)_{n\in\N}$ and $(\theta_n)_{n\in\N}$ converge to some constants $T_{\infty}^*$ and $\theta_{\infty}^*$ respectively. For the sake of convenience, we write $T^*=T^*_\infty$ and $\theta^*=\theta_\infty^*$. In the four theorems below we allow for $k \in [1,\infty)$, although $k\in\N$ is needed to interpret the constraint in terms of a subgraph density.

\subsubsection{Parameter regime}
Our first two theorems identify the \emph{parameter regime} for BEE. 

\begin{theorem}{\bf [Computational criterion for ensemble equivalence]}
\label{thm:subgraph}
For $\theta \in [0,\infty)$ and $k \in [1,\infty)$, let $u^*(\theta,k)$ be a maximiser of
\begin{equation}
\label{eq:variationalformulachatterjee}
\sup_{u \in [0,1]} \,\big[\theta u^k-I(u)\big].
\end{equation}
(a) For every $T^* \in [(\tfrac12)^k,1)$ there is ensemble equivalence if and only if there exists a $\theta_0 = \theta_0(T^*,k) \in [0,\infty)$ such that $(u^*(\theta_0,k))^k = T^*$. In that case the Lagrange multiplier $\theta^*=\theta^*(T^*,k)$ equals $\theta_0$.\\ 
(b) There exists a unique $\that = \that(k) \in [0,\infty)$ such that $\theta^*(T^*,k) = \that$ for all $T^*$ for which there is breaking of ensemble equivalence.
\end{theorem}

\begin{figure}[htbp]
\centering
\hspace{1.5cm}\includegraphics[scale=0.6]{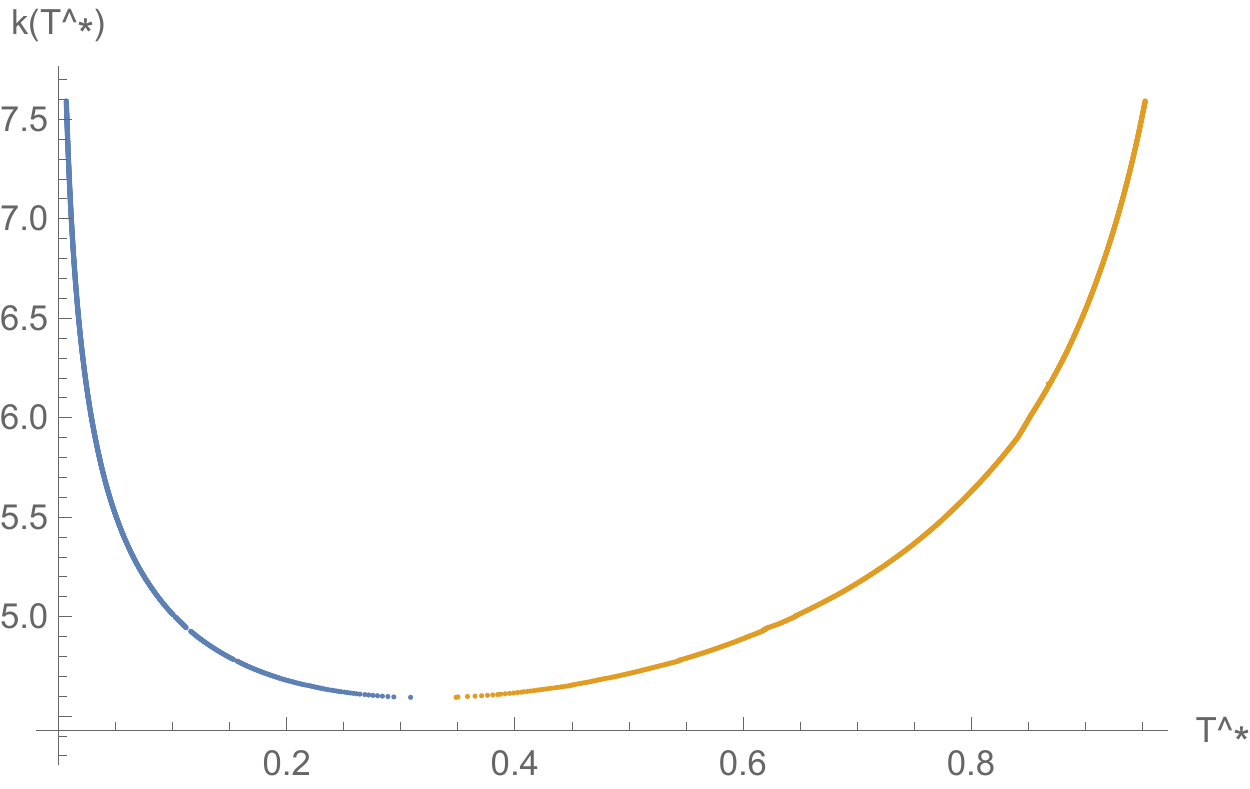}
\caption{\small A numerical picture of the phase diagram. The blue and orange lines together form the critical curve $(T^*,k_c(T^*))$. In the figure, $T^*$ is denoted by \texttt{T\^{}*} and $k_c(T^*)$ is denoted by \texttt{k(T\^{}*)}. The minimum is achieved at $k_0 = 4.591\ldots$ and $T_0 = 0.3237\ldots$.} 
\label{fig-phase}
\end{figure}

\begin{theorem}{\bf [Phase diagram]}
\label{thm:phase}
$\mbox{}$\\
(a) There exists a function $k_c \colon\, (0,1) \to [1,\infty)$ such that for every $T^*\in(0,1)$ there is ensemble equivalence when $\log_2(1/T^*) \leq k \leq k_c(T^*)$ and breaking of ensemble equivalence when $k>k_c(T^*)$. \\
(b) $T^* \mapsto k_c(T^*)$ achieves a unique minimum at the point $(T_0,k_0)$, with $k_0$ the unique solution of the equation $\frac{k_0-1}{k_0} \log(k_0-1) = 1$ and $T_0 = (\frac{k_0-1}{k_0})^{k_0}$.\\
(c) $T^* \mapsto k_c(T^*)$ is analytic on $(0,1) \setminus \{T_0\}$.\\ 
(d) $(\frac{1}{2})^{k_c(T^*)} \sim T^*$ as $T^* \downarrow 0$ and $k_c(T^*)(\frac{1}{2})^{k_c(T^*)} \sim 1-T^*$ as $T^* \uparrow 1$.
\end{theorem}
Note that the results above only hold in the regime $T^* \in [(\tfrac12)^k,1)$, which corresponds to the regime $\theta^*\geq0$. This assumption is necessary, since the results from \cite{CD13} that we use only hold for non-negative $\theta^*$. For $k$ and $T^*$ not in this regime, we do not know if there is ensemble equivalence.

\subsubsection{Replica symmetry}
Our last two theorems quantify the specific relative entropy and the spectral gap in the \emph{replica symmetry regime}. This regime was first defined in \cite{CV11} and further studied in \cite{LZ16}. Using the theory developed in \cite{LZ16}, it is possible to quantify the specific relative entropy $s_\infty$ and the difference of the largest eigenvalue $\delta_\infty$ for certain $T^*$ in the BEE-phase.

\begin{definition}{\bf[Replica symmetry]}
{\rm
Consider the Erd\H{o}s-R\'enyi random graph $G$ on $n$ vertices with retention probability $p \in [0,1]$ conditioned on $t(F,G)\geq T^*$ for some finite simple graph $F$. If $G$ converges in the cut metric to a constant graphon, then we say that $T^*$ is in the \emph{replica symmetric region}.}
\end{definition}

From the theory of large deviations for random graphs developed in \cite{CV11}, we know that $T^*$ is in the replica symmetric region if and only if
\begin{equation}
\label{rep1}
\inf_{t(F,f)\geq T^*}I_p(f)
\end{equation}
is minimised by a constant graphon, with $I_p$ the rate function given by
\begin{equation}
\label{rep2}
I_p(f) = \int_{[0,1]^2}\dx\dy\left(f(x,y)\,\log\frac{f(x,y)}{p}+[1-f(x,y)]\,\log\frac{1-f(x,y)}{1-p}\right).
\end{equation}
Note that $I(f) = I_{\frac{1}{2}}(f)-\frac{1}{2}\log2$. Hence, if $T^*$ is in the replica symmetric region, then there is an explicit solution for the second supremum in \eqref{varreprsinfty}. In \cite{LZ16}, it was shown that $T^*$ is in the replica symmetric region when $(T^*,I_p(T^{*\,1/d}))$ lies on the convex minorant of the function $x\mapsto I_p(x^{1/d})$, with $d$ the maximum degree of the subgraph $F$. If $F$ is regular, then the converse statement holds as well.

Fix a subgraph $F$ with $k$ edges and maximum degree $d$. Let 
\begin{equation}
T_1^*(k) \in ((\tfrac12)^k,T_0), \qquad T_2^*(k) \in (T_0,1),
\end{equation}
be the two solutions of the equation $k_c(T^*(k)) = k$, so that
\begin{equation}
(T_1^*(k),T_2^*(k)) =  \text{BEE-phase}.
\end{equation}
In Lemma \ref{lemma:convexminorantEE}, we prove that the replica symmetric region contains $[(\frac{1}{2})^k,T_1^*(d)] \cup [T_2^*(d),1]$. Thus, if $d<k$, then in part of the BEE-phase there is replica symmetry. This allows us to formulate the following two theorems (which are vacuous for $d=k$). 

\begin{theorem}{\bf [Specific relative entropy]}
\label{thm:entropy}
For every $T^*$ in the replica symmetric part of the phase of breaking of ensemble equivalence,
\begin{equation}
s_\infty = \left\{\begin{array}{ll}
\that(k)\,[T_1^*(k)-T^*] + \big[I(T^{*\,1/k})-I(T_1^*(k)^{1/k})\big] > 0, &T^* \in (T_1^*(k),T_1^*(d)],\\[0.2cm]
\that(k)\,[T_2^*(k)-T^*] + \big[I(T^{*\,1/k})-I(T_2^*(k)^{1/k})\big] > 0, &T^* \in [T_2^*(d),T_2^*(k)).
\end{array}
\right.
\end{equation}
Consequently,
\begin{equation}
s_\infty = \left\{\begin{array}{ll}
C(T_1^*(k),k)\,[T^*-T_1^*(k)]^2 + O([T-T_1^*(k)]^3),  &T^* \downarrow T_1^*(k),\\[0.2cm]
C(T_2^*(k),k)\,[T^*-T_2^*(k)]^2 + O([T-T_2^*(k)]^3),  &T^* \uparrow T_2^*(k),
\end{array}
\right.
\end{equation}
with
\begin{equation}
C(T^*,k) = \frac{T^{*\,(1-2k)/k}}{2k}\left\{\frac{1}{k}\left(1+\frac{T^{*\,1/k}}{1-T^{*\,1/k}}\right)+\left(\frac{1}{k}-1\right)\log\left(\frac{T^{*\,1/k}}{1-T^{*\,1/k}}\right)\right\}.
\end{equation}
\end{theorem}


\begin{theorem}{\bf [Spectral signature]}
\label{thm:gap}
For every $T^*$ in the replica symmetric part of the phase of breaking of ensemble equivalence,
\begin{equation}
\begin{aligned}
&\delta_\infty = \frac{T^{*\,1/k}_1\,[T^*_2(k)-T^*] + T^{*\,1/k}_2\,[T^*-T^*_1(k)]}
{T^*_2(k)-T^*_1(k)}-T^{*\,1/k}<0,\\
&T^* \in (T^*_1(k),T^*_1(d)] \cup [T^*_2(d),T^*_2(k)).
\end{aligned}
\end{equation}
Consequently, 
\begin{equation}
\delta_\infty = \left\{\begin{array}{ll}
\hat{C}(T^*_1(k),k)\,[T^*-T^*_1(k)] + O([T^*-T^*_1(k)]^2), & T^*\downarrow T^*_1(k),\\[0.2cm]
\hat{C}(T^*_2(k),k)\,[T^*-T^*_2(k)] + O([T^*-T^*_2(k)]^2), & T^*\uparrow T^*_2(k),
\end{array}
\right.
\end{equation}
with
\begin{equation}
\hat{C}(T^*,k) = \frac{T^{*\,1/k}_2(k)-T^{*\,1/k}_1(k)}{T^*_2(k)-T^*_1(k)}-\frac{1}{k} T^{*\,(1-k)/k}.
\end{equation}
\end{theorem}

\begin{figure}[htbp]
\centering
\hspace{1.5cm}\includegraphics[scale=0.8]{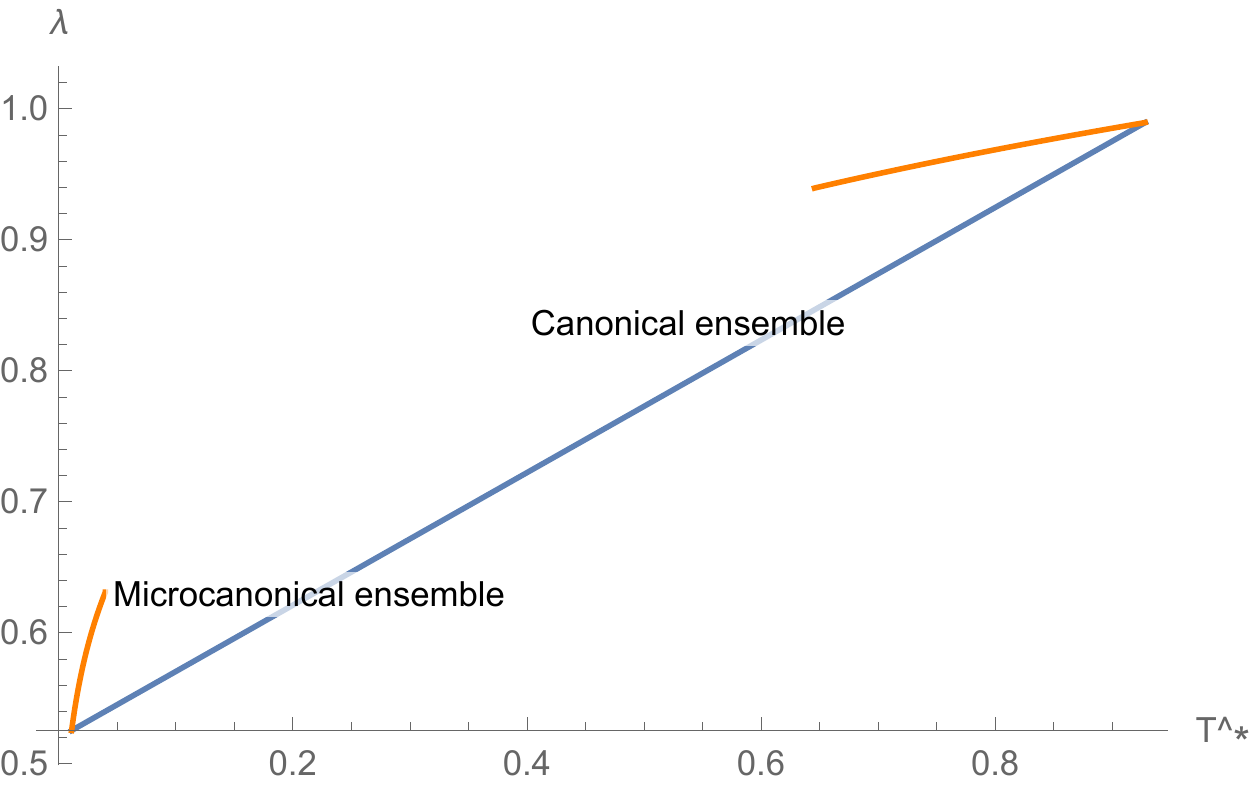}
\caption{\small A numerical picture of the average largest eigenvalue $\lambda=\limn\frac{1}{n}\mathbb{E}[\lambda_n]$ of the adjacency matrix under the microcanonical ensemble (top curve) and the canonical ensemble (bottom curve), as a function of $T^*$ for a subgraph $F$ with $k=7$ edges and maximum degree $d=5$. The top curve is shown only for $T^*$ in the replica symmetric region. In the region of replica symmetry breaking we have no explicit expression for $\lambda$ under the microcanonical ensemble.}
\label{fig-eigenvalue}
\end{figure}


\subsection{Typical graph under the microcanonical and canonical ensemble}
\label{ss.typgraph}

The BEE-phase can also be characterised through convergence of the random graph drawn from the two ensembles. In Lemmas \ref{lemma:convergencecanonicalminimisers} and \ref{lemma:convergencemicrocanonical} below we show that the random graph drawn from the canonical ensemble converges to the maximiser(s) of the first supremum of \eqref{varreprsinfty}, while the random graph drawn from the microcanonical ensemble converges to the maximiser(s) of the second supremum of \eqref{varreprsinfty}. 

Outside the BEE-phase, both suprema are attained by the constant graphon $h\equiv T^{*\,1/k}$, meaning that for large $n$ both ensembles behave approximately like the Erd\H{o}s-R\`enyi random graph with retention probability $p=T^{*\,1/k}$. Inside the BEE-phase, the first supremum is maximised by the two constant graphons $T_1^*(k)^{1/k}$ and $T_2^*(k)^{1/k}$, neither of which lies in $\Wt^*$. Consequently, the random graph drawn from the canonical ensemble converges to the random graphon
\begin{equation}
\frac{T_2^*(k)-T^*}{T_2^*(k)-T_1^*(k)}\,\delta_{T_1^*(k)^{1/k}}
+ \frac{T^*-T_1^*(k)}{T_2^*(k)-T_1^*(k)}\,\delta_{T_2^*(k)^{1/k}},
\end{equation}
meaning that for large $n$ the canonical ensemble behaves approximately like a \emph{mixture} of two Erd\H{o}s-R\'enyi random graphs. If $T^*$ is in the replica symmetric part of the BEE-phase, then the second supremum is still minimised by the constant graphon $h \equiv T^{*\,1/k}$. Hence, the random graph is asymptotically deterministic under the microcanonical ensemble and random under the canonical ensemble. Thus, BEE occurs due to \emph{coexistence} of two densities. This is similar in spirit to the coexistence of water and ice at the melting point, at which a first-order phase transition between water and ice occurs.   

In the region of replica symmetry breaking, the maximisers of the second supremum are unknown, and it is not even known whether or not there is a unique maximiser. In case of non-uniqueness, also under the microcanonical ensemble the random graph is asymptotically random.


\subsection{Discussion and outline}
\label{ss.disc}

{\bf 1.}
Theorem~\ref{thm:subgraph} reduces the variational formula on $\Wt$ to a variational formula on $[0,1]$, and is an application of a reduction principle explained in \cite{CD13} (see also \cite{C15}). The proof relies on the variational characterisation in Theorem \ref{th:Limit}. The main difficulty lies in computing the tuning parameter $\theta^*$ as a function of the density $T^*$, which is resolved through Lemma \ref{lemma:tuningparameter}. The proof follows from an analysis of the two variational expressions, for which we rely in part on the results in \cite{RY13}. From Theorem~\ref{thm:subgraph}, for each $k$ we can identify the BEE-phase as follows. The expression in \eqref{eq:variationalformulachatterjee} has at most two local maximisers $u_1^*(\theta)<u_2^*(\theta)$, which are both increasing in $\theta$. For $\theta<\that$, $u_1^*(\theta)$ is the global maximiser, for $\theta>\that$, $u_2^*(\theta)$ is the global maximiser, and for $\theta=\that$, $u_1^*(\theta)$ and $u_2^*(\theta)$ are both global maximisers. Hence, the values $u\in(u_1^*(\theta),u_2^*(\theta))$ can never be a global maximiser, and so the BEE-phase contains $(u_1^*(\theta)^k,u_2^*(\theta)^k)$. Since $u_1^*(0)=\frac{1}{2}$ and $\lim_{\theta\rightarrow\infty}u_2^*(\theta)=1$, the interval $(u_1^*(\theta)^k,u_2^*(\theta)^k)$ is the entire BEE-phase.

\medskip\noindent
{\bf 2.}
Theorem~\ref{thm:phase} identifies the BEE-phase and captures the main properties of the critical curve bordering this phase. The proof relies on Lemma \ref{lemma:convexminorantEE} below, which allows us to use results from \cite{LZ16} and establish a connection between ensemble equivalence and replica symmetry, in the sense that $T^*$ lies in the BEE-phase for a subgraph with $k$ edges if and only if $T^*$ lies in the region of replica symmetry breaking for $p=\frac12$ and a $k$-regular subgraph (recall \eqref{rep1}--\eqref{rep2}). This connection is purely analytic: it establishes equivalence of variational formulas and implies that the graph in Figure \ref{fig-phase} is a cross-section of the curves in \cite[Figure 2]{LZ16} at $p=\frac12$. It is not clear, however, how to probabilistically interpret the relationship between replica symmetry for regular subgraphs and breaking of ensemble equivalence for general graphs. Note that we do not require any regularity of the subgraph $F$, and also the degrees of $F$ do not play any role. It might be easier to use the variational formula in \eqref{eq:variationalformulachatterjee} (with $I_p$ instead of $I$) to analyse replica symmetry, rather than the convex minorant of $ x \mapsto I_p(x^{1/k})$.

\medskip\noindent
{\bf 3.}
Theorem~\ref{thm:entropy} gives an explicit formula for the specific relative entropy $s_\infty$ in part of the BEE-phase. The proof exploits the connection with replica symmetry. If a subgraph has more edges than its maximal degree (i.e., is not a $k$-star), then the BEE-phase near $T_1^*(k)$ and $T_2^*(k)$ is replica symmetric. This implies that the second supremum in \eqref{varreprsinfty} also has a constant maximiser, which allows us to explicitly compute $s_\infty$. It turns out that the relative entropy undergoes a second-order phase transition as $T^*$ approaches the critical curve. 

\medskip\noindent
{\bf 4.}
Theorem~\ref{thm:gap} shows that the working hypothesis put forward in \cite{DGHM20} is met in the replica symmetric part of the BEE-phase. A random graph drawn from the canonical ensemble converges to a constant graphon whose height is a random mixture of the two maximisers $u_1,u_2$ of \eqref{eq:variationalformulachatterjee}. The average largest eigenvalue converges to a value on the line segment connecting $(u_1^{1/k},u_1)$ and $(u_2^{1/k},u_2)$. In the region of replica symmetry, a random graph drawn from the microcanonical ensemble converges to the constant graphon whose height is $(T^*)^{1/k}$, as illustrated in Figure \ref{fig-eigenvalue}. Note that the average largest eigenvalue is larger in the microcanonical ensemble than in the canonical ensemble, contrary to what was found in \cite{DGHM20}, where the constraint was on the degree sequence. It turns out that $\delta_\infty$ undergoes a first-order phase transition as $T^*$ approaches the critical curve. 

\medskip\noindent
{\bf 5.} The numerical picture of the phase diagram in Figure \ref{fig-phase} was made using Mathematica. The computations involve finding an approximate value of $\that(k)$ for each $k$ (up to an accuracy of 5 digits), and computing $u_1^*(\that(k),k)$ and $u_2^*(\that(k),k)$. The dotted lines are formed by the points $(u_1^*(k)^k,k)$ and $(u_2^*(k)^k,k)$. This is done for $k$ starting at 4.592 and increasing with increments of 0.002.

\medskip\noindent
{\bf 6.}
In \cite{T15}, BEE for interacting particle systems was studied at three different levels: thermodynamic, macrostate and measure. It was shown that these levels are in fact equivalent. A general formalism was put forward, based on an abstract large deviation principle, linking the occurrence of BEE to non-convexity of the rate function associated with the microcanonical ensemble as a function of the parameters controlling the constraint. In our context, the large deviation principle for graphons in \cite{CV11} provides the conceptual basis for identifying the BEE-phase via the variational formula derived in \cite{dHMRS18}, and the link with the convex minorant mentioned in item 2 fits in with the picture provided in \cite{T15}.  

\medskip\noindent
{\bf Outline.}
The remainder of the paper is organised as follows. Theorems~\ref{thm:subgraph}--\ref{thm:gap} are proved in Sections~\ref{s.proofmain}--\ref{s.proofgap}, respectively.


\section{Proof of Theorem~\ref{thm:subgraph}}
\label{s.proofmain}

Throughout the proof, we fix $k\in\N$, and suppress $k$ from the notation. We analyse the expression
\begin{equation}
\label{eq:variationalformularelativeentropy}
\sup_{\h \in \W} \,[\theta T(\h) - I(\h)\big]
\end{equation}
with $\theta \in [0,\infty)$, and determine for which values of $T^*$ a maximiser of this supremum is in the set $\Wt^*$. Note that it suffices to consider $\theta\in[0,\infty)$, since $T^*\geq(\frac{1}{2})^k$. This was shown in \cite[Lemma 5.1]{dHMRS18} in the case that $F$ is a triangle, but the proof generalizes to general finite simple graphs.

By \cite[Theorem 4.1]{CD13}, the supremum equals the supremum in \eqref{eq:variationalformulachatterjee}, and each maximiser of \eqref{eq:variationalformularelativeentropy} is a constant function, where the constant is a maximiser of \eqref{eq:variationalformulachatterjee}. Furthermore, by Lemma \ref{lemma:tuningparameter}, $\theta^*$ is a maximiser of the supremum
\begin{equation}
\sup_{\theta \geq 0} \,\big[\theta T^* - \theta T(u^*(\theta)) + I(u^*(\theta))\big] 
= \sup_{\theta \geq 0}\,\big[\theta T^* - \theta (u^*(\theta))^k + I(u^*(\theta))\big],
\end{equation}
where $u^*(\theta)$ is a maximiser of \eqref{eq:variationalformulachatterjee}. By \cite[Proposition 3.2]{RY13}, $l_\theta(u) := \theta u^k - I(u)$ has at most 2 maxima and there exists a $\hat{\theta}$ such that, for $\theta<\hat{\theta}$, the first local maximum is the unique global maximum and, for $\theta>\hat{\theta}$, the second local maximum is the unique global maximum. Hence, for all $\theta\neq\hat{\theta}$, $u^*(\theta)$ is well-defined. For $\theta=\that$, both maxima are a global maximum. In that case, we let $u^*(\theta)$ denote either of the two maximisers.

Let $m(\theta) = \theta T^* - l_\theta(u^*(\theta)) =\theta T^*-\theta(u^*(\theta))^k + I(u^*(\theta))$. In Figure \ref{fig-variational formula}, plots of $l_\theta$ are shown for several values of $\theta$. Write $u:=u^*(\theta)$ and $u':=\frac{\partial u}{\partial\theta}(\theta)$. Then [kijk nog even naar]
\begin{equation}
\begin{split}
\label{eq:lderivative}
l_\theta'(u) = \theta ku^{k-1} - \tfrac{1}{2}\log u + \tfrac{1}{2}\log(1-u) = 0
\end{split}
\end{equation}
and
\begin{equation}
\begin{split}
m'(\theta) 
&= T^*-u^k-\theta ku^{k-1}u' + \tfrac{1}{2}u'\log(u)-\tfrac{1}{2}u'\log(1-u)\\
&= T^* - u^k - u'(\tfrac{1}{2}\log u-\tfrac{1}{2}\log(1-u)-\theta ku^{k-1})\\
&= T^*-u^k.
\end{split}
\end{equation}
Hence, if there exists a $\theta_0\geq0$ such that $(u^*(\theta_0))^k=T^*$, then $m'(\theta_0)=0$, and so $\theta^*=\theta_0$. In that case $(u^*(\theta^*))^k=T^*$, so there is ensemble equivalence. If such a $\theta_0$ does not exist, then there is breaking of ensemble equivalence.

Let $u_1^*(\theta)$ and $u_2^*(\theta)$ be the first and second local maximum of $l_\theta$, respectively. Then $\theta \mapsto u_1^*(\theta)$ and $\theta \mapsto u_2^*(\theta)$ are increasing. Furthermore, for all $\theta<\that$, $u_1^*(\theta)$ is the unique global maximum, while for all $\theta>\that$, $u_2^*(\theta)$ is the unique global maximum. Hence, if there is breaking of ensemble equivalence, then $m'(\theta)>0$ for $\theta<\that$ and $m'(\theta)<0$ for $\theta>\that$. We conclude that $\theta^*=\that$.

\begin{figure}[htbp]
\centering 
\includegraphics[scale=0.35]{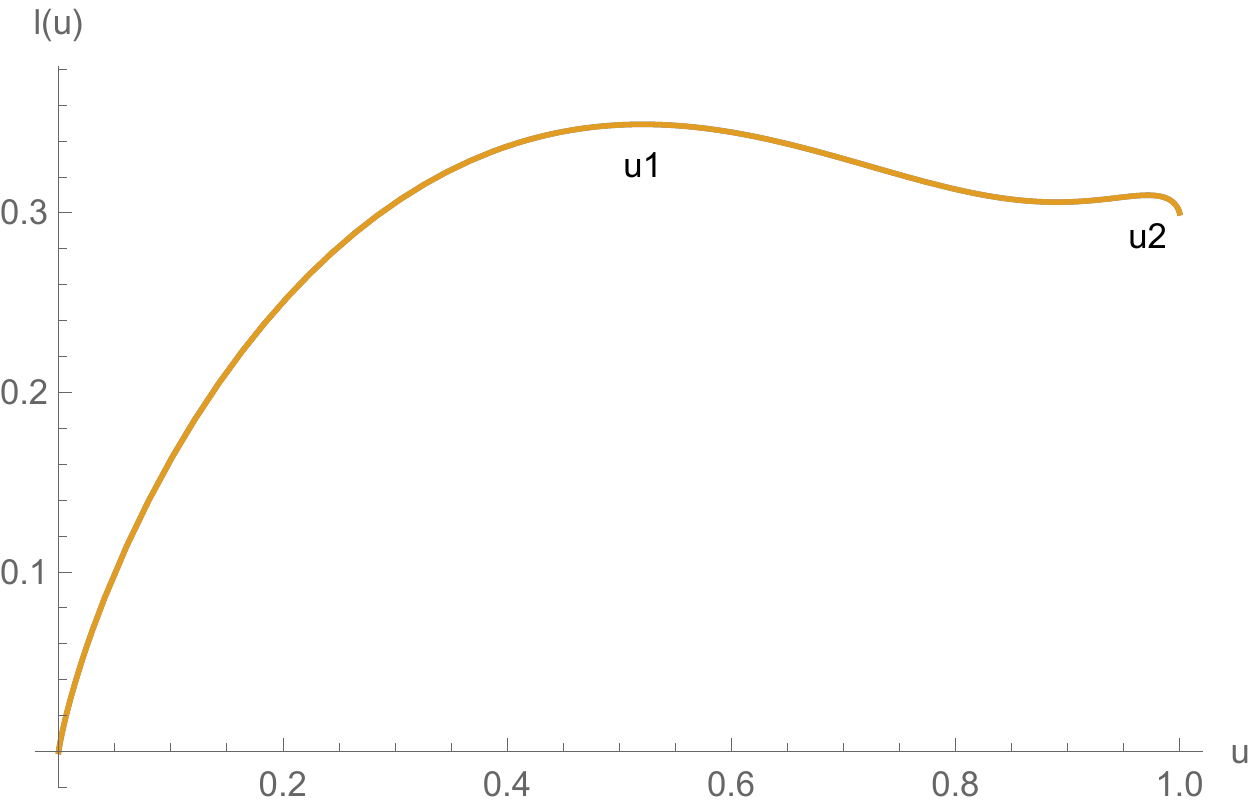}
\quad
\includegraphics[scale=0.35]{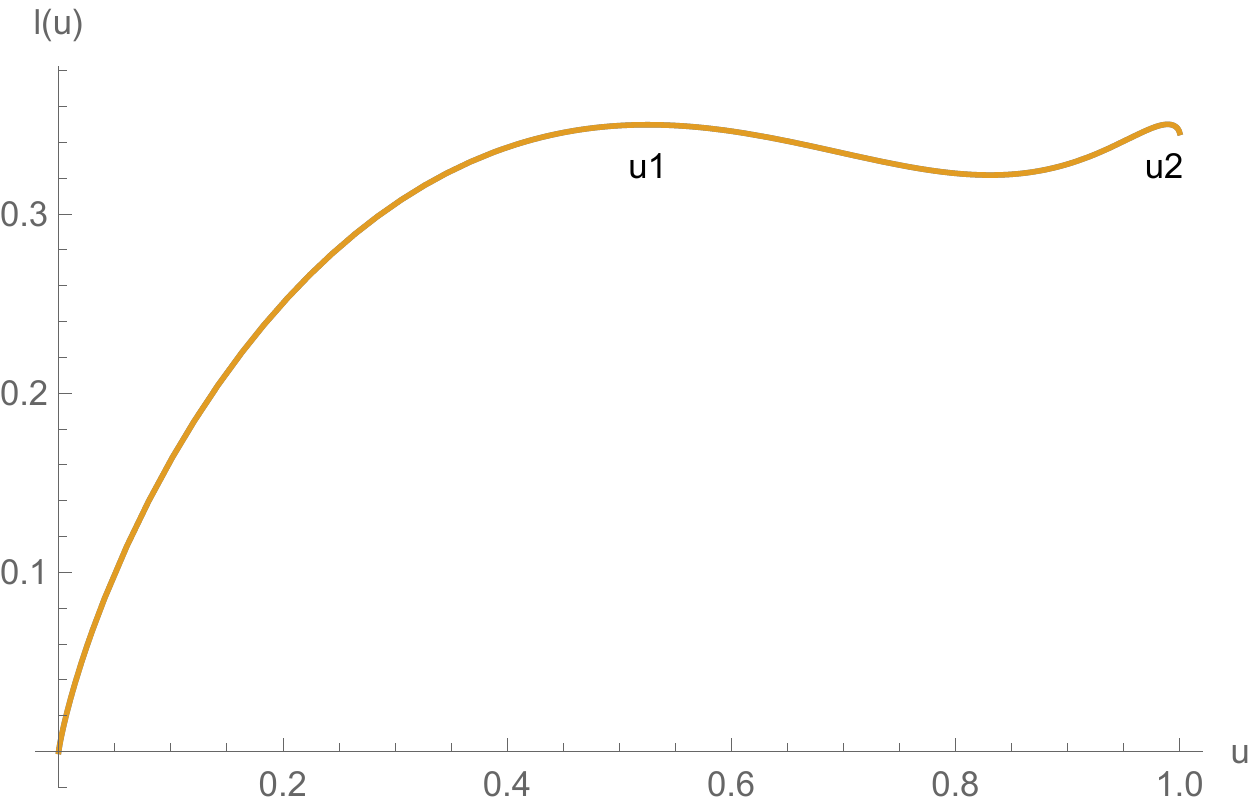}
\quad
\includegraphics[scale=0.35]{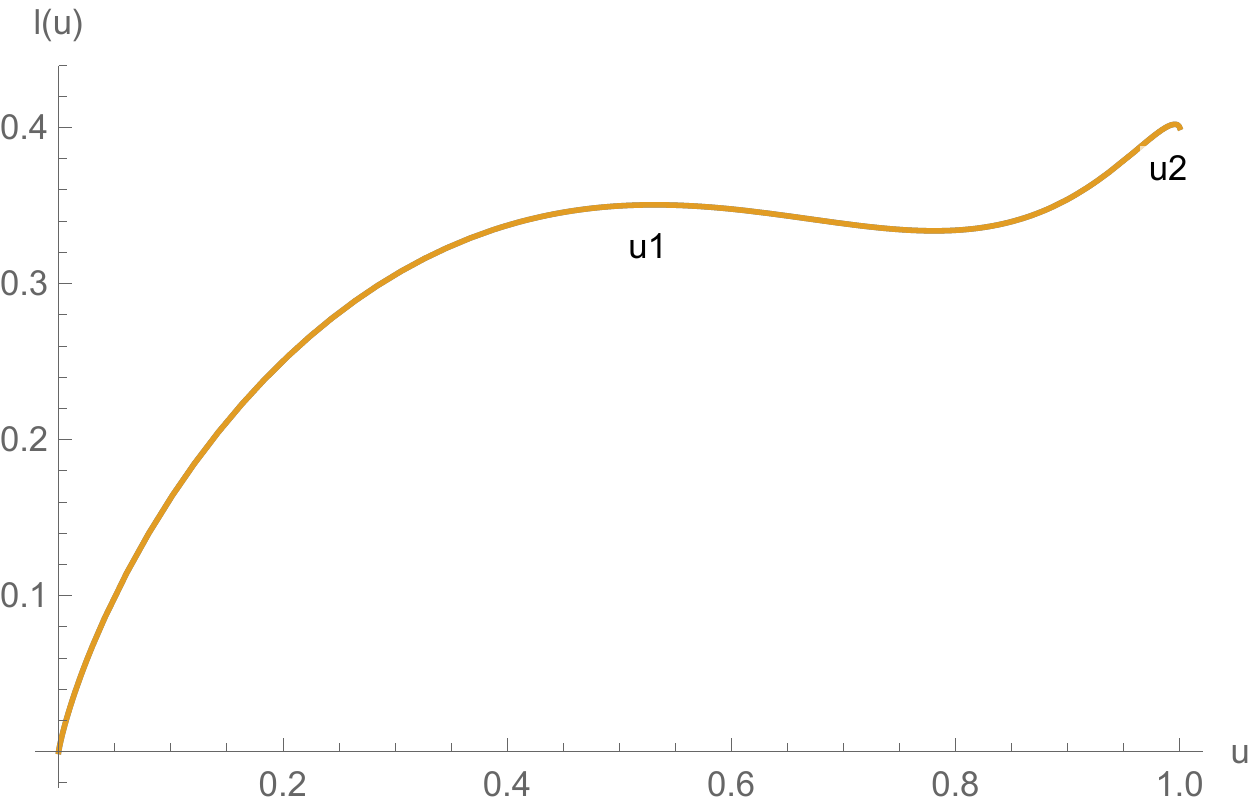}
\caption{\small Three plots of $l_{\theta}(u)$ for $k=7$ and $\theta=0.3$, $\theta=\that(7)$ and $\theta=0.4$, respectively. For $\theta=0.3$, $u_1^*(\theta)$ is the global maximiser, for $\theta=\that(7)$, $u_1^*(\theta)$ and $u_2^*(\theta)$ are both global maximisers, and for $\theta=0.4$, $u_2^*(\theta)$ is the global maximiser. In the figures, the function $l_\theta(u)$ is denoted by \texttt{l(u)}. The local maximisers $u_1^*(\theta)$ and $u_2^*(\theta)$ are denoted by \texttt{u1} and \texttt{u2} respectively. The BEE-phase is $(u_1^*(\that)^k,u_2^*(\that)^k)$.}
\label{fig-variational formula}
\end{figure}


\section{Proof of Theorem~\ref{thm:phase}}
\label{s.proofphase}

We first fix some notation. For given $k$ and $\theta$, let $u_1^*(\theta,k)$ and $u_2^*(\theta,k)$ be the first and second local maximum respectively of $l_{\theta,k}(u)=\theta u^k-I(u)$. Let $\that(k)$ be the unique value of $\theta$ such that $u_1^*(\that(k),k)=u_2^*(\that(k),k)$. Define $J_k(x)=I(x^{1/k})$ and $T_1(k)=u_1^*(\that(k),k)^k$, $T_2(k)=u_2^*(\that(k),k)^k$.

\paragraph{Existence of $k_c$.}

Lemmas~\ref{lemma:convexminorantEE}--\ref{lemma:mono} below establish the existence of the critical curve. Lemma \ref{lemma:convexminorantEE} shows the connection between replica symmetry and ensemble equivalence as discussed in Section \ref{ss.disc}, since $T$ is in the region of replica symmetry for $p=\frac{1}{2}$ if and only if $(T,I(T^{1/k}))$ lies on the convex minorant of $J_k$.

\begin{lemma}{\bf [Connection with replica symmetry]}
\label{lemma:convexminorantEE}
Let $k\geq 1$ and $T\in[(\frac{1}{2})^k,1)$. There is ensemble equivalence for $T^*=T$ if and only if $(T,I(T^{1/k}))$ lies on the convex minorant of the function $J_k$.
\end{lemma}

\begin{proof}
Note that $I(x)=I_{1/2}(x)-\frac{1}{2}\log2$ (recall \eqref{rep2}), so $(T,I(T^{1/k}))$ lies on the convex minorant of $J_k$ if and only if $(T,I_{1/2}(T^{1/k}))$ lies on the convex minorant of the function $x\mapsto I_{1/2}(x^{1/k})$. 

In \cite[Appendix A]{LZ16}, it is shown that there exist $q_1,q_2\in(0,1)$ such that $(q^k,I(q))$ is not on the convex minorant of $J$ if and only if $q^k\in(q_1^k,q_2^k)$. The values $q_1,q_2$ are defined as the unique values in $[0,1]$ such that the tangent lines of $J$ at $q_1^k$ and $q_2^k$ are the same, i.e., $J'(q_1^k)=J'(q_2^k)=:D$ and $J(q_1^k)+D(q_2^k-q_1^k)=J(q_2^k)$, or equivalently, $Dq_1^k-J(q_1^k)=Dq_2^k-J(q_2^k)$.

Recall from Section \ref{ss.disc} that there is breaking of ensemble equivalence for $T^*=T\in[(\frac{1}{2})^k,1)$ if and only if $T\in(u_1^k,u_2^k)$, where $u_1=u_1^*(\that(k),k)$ and $u_2=u_2^*(\that(k),k)$. Since $u_1,u_2$ are the maximisers of $x\mapsto \that x^k-I(x)$ and $x\mapsto x^k$ is monotone, we have that $T_1:=u_1^k$ and $T_2:=u_2^k$ are the maximisers of $x\mapsto \that x-I(x^{1/k})=\that x-J(x)$. Hence, $J'(T_1)=J'(T_2)=\that$. Furthermore, $\that$ was defined such that $\that u_1^k-I(u_1)=\that u_2^k-I(u_2)$, so $\that T_1-J(T_1)=\that T_2-J(T_2)$.

From the above, we conclude that $u_1=q_1$ and $u_2=q_2$. This completes the proof.
\end{proof}

There is ensemble equivalence for $T^*\leq T_1(k)$ and $T^*\geq T_2(k)$, and ensemble inequivalence for $T^*\in(T_1(k),T_2(k))$. By \cite[Lemma A.5]{LZ16}, $k \mapsto u_1^*(\that,k)$ is decreasing and $k \mapsto u_2^*(\that,k)$ is increasing. Although $k\mapsto (u_1^*(\that,k))^k$ is clearly decreasing, it is not a priori obvious whether $k \mapsto (u_2^*(\that,k))^k$ is increasing, since $u_2^*(\that,k)<1$. If the latter is the case, then for all $k>k_c(T^*)$ there is breaking of ensemble equivalence, and for all $k\leq k_c(T^*)$ there is ensemble equivalence, where $k_c(T^*)$ is chosen such that $T^*=T_1(k_c)$ or $T^*=T_2(k_c)$. This proves the first part of Theorem \ref{thm:phase}. Also, since $T_1(k)\geq(\frac{1}{2})^{k}$, this also shows that $k_c\geq\log_2(1/T^*)$. The following lemma fills in the gap.

\begin{lemma}{\bf [Monotonicity]}
\label{lemma:mono}
The function $k \mapsto T_1(k)$ is decreasing and $k \mapsto T_2(k)$ is increasing.
\end{lemma}

\begin{proof}
The function $\frac{\partial J_k}{\partial k}$ is a concave function for every $k$. Because the line segment connecting $(T_1(k),J_k(T_1(k)))$ with $(T_2(k),J_k(T_2(k)))$ lies below the curve $(x,J_{k}(x))$, we have that, for all $\alpha\in[0,1]$ and $k' \downarrow k$,
\begin{equation}
\begin{split}
&J_{k'}\big(\alpha T_1(k)+(1-\alpha)T_2(k)\big)\\
&= J_k\big(\alpha T_1(k)+(1-\alpha)T_2(k)\big)+(k'-k)\tfrac{\partial J_k}{\partial k}\big(\alpha T_1(k)+(1-\alpha)T_2(k)\big)+o(k'-k)\\
&\geq \alpha J_k(T_1(k))+(1-\alpha)J_k(T_2(k))+(k'-k)(\alpha\tfrac{\partial J_k}{\partial k}(T_1(k))+(1-\alpha)\tfrac{\partial J_k}{\partial k}(T_2(k)))+o(k'-k)\\
&= \alpha J_{k'}(T_1(k))+(1-\alpha)J_{k'}(T_2(k))+o(k'-k).
\end{split}
\end{equation}
Hence, for $k'>k$ small enough, the line segment connecting the points $(T_1(k),J_{k'}(T_1(k)))$ and $(T_2(k),J_{k'}(T_2(k)))$ lies below the curve $(x,J_{k'}(x))$, and is not tangent to the curve at any of the end points. Thus, by \cite[Lemma A.3]{LZ16}, $T_1(k')<T_1(k)<T_2(k)<T_2(k')$.
\end{proof}


\paragraph{Minimum of $k_c$.}

By \cite[Proposition 3.2]{RY13}, for all $k \leq k_0$, $l_{\theta,k}$ has a unique maximiser for all $\theta\geq0$. For all $k>k_0$, there exist a $\theta \geq 0$ such that $l_{\theta,k}$ has two maximisers. Hence, the minimum value of $k_c(T^*)$ is $k_0$. In the proof of \cite[Proposition 3.2]{RY13} it is shown that $\that(k_0)=\frac{k_0^{k_0-1}}{2(k_0-1)^{k_0}}$, and so
\begin{equation}
\begin{split}
l'_{\that(k_0),k_0}(\tfrac{k_0-1}{k_0}) = \tfrac{(k_0)^{k_0-1}}{2(k_0-1)^{k_0}}\,k_0\left(\tfrac{k_0-1}{k_0}\right)^{k_0-1}-\tfrac{1}{2}\log(k_0-1)
= \tfrac{k_0}{2(k_0-1)}-\tfrac{1}{2}\log(k_0-1) = 0.
\end{split}
\end{equation}
Hence, $u^*(\that(k_0),k_0)=\frac{k_0-1}{k_0}$, and so for $T^*=(\frac{k_0-1}{k_0})^{k_0}$ we have $k_c(T^*)=k_0$. We conclude that $k_c$ has a unique minimum at the point $((\frac{k_0-1}{k_0})^{k_0},k_0)$.


\paragraph{Analyticity of $k_c$.}

Analyticity of $k_c$ follows from a straightforward application of the implicit function theorem. Let $f\colon\,(0,\infty)\times(0,1)^2\to\R^2$ be given by
\begin{equation}
 f(k,x,y) = \big(J_k'(x)-J'_k(y),J'_k(x)x-J'_k(y)y+J(y)-J(x)\big).
\end{equation}
Recall from the proof of Lemma \ref{lemma:convexminorantEE} that, for each $k$, $T_1(k)$ and $T_2(k)$ are defined such that $f(k,T_1(k),T_2(k)) = 0$. Note that $f$ is analytic, and its Jacobian
\begin{equation}
\begin{split}
\begin{pmatrix}
\frac{\partial f_1}{\partial x} & \frac{\partial f_1}{\partial y}\\
\frac{\partial f_2}{\partial x} & \frac{\partial f_2}{\partial y}
\end{pmatrix}
(T_1(k),T_2(k)) =
\begin{pmatrix}
J''_k(T_1(k)) & -J''_k(T_2(k))\\
T_1(k)J_k''(T_1(k)) & -T_2(k))J_k''(T_2(k))
\end{pmatrix},
\end{split}
\end{equation}
is invertible if $T_1(k)\neq T_2(k)$. Hence, for all $k>k_0$, $T_1$ and $T_2$ are analytic functions of $k$, so $k_c$ is an analytic function of $T^*$ outside its minimum.

Next, consider the behaviour of $k_c$ near $T_0$, so as $T_2-T_1\downarrow0$. By implicit differentiation, as $k\downarrow k_0$, the derivative of $T_1(k)$ is given by
\begin{equation}
\begin{split}
T_1'(k) 
&= \frac{1}{(T_1-T_2)J_k''(T_1)J_k''(T_2)}\left[(T_2-T_1)J_k''(T_2)\frac{\partial J'_k}{\partial k}
(T_1)+J_k''(T_2)\left(\frac{\partial J_k}{\partial k}(T_1)-\frac{\partial J_k}{\partial k}(T_2)\right)\right]\\
&= \frac{1}{J_k''(T_1)}\left(\frac{\partial J_k'}{\partial k}(T_1)+\frac{\frac{\partial J_k}{\partial k}(T_1)
-\frac{\partial J_k}{\partial k}(T_2)}{T_2-T_1}\right)\\
&= \frac{1}{J_k''(T_1)}O(T_2-T_1).
\end{split}
\end{equation}
It is not difficult to show that, for $k=k_0$, the function $J_{k_0}''$ has a zero that is also a minimum at $T=T_0$. Hence, as $k\downarrow k_0$, $J_k''(T_1(k))=O((T_2-T_1)^2)$, which implies that the derivative of $T_1'(k)$ diverges as $k\downarrow k_0$. In a similar fashion, we can show that the derivative of $T_2'(k)$ diverges as $k\downarrow k_0$. Hence, at $T_0$, $k_c$ is at least differentiable and has derivative zero.


\paragraph{Scaling of $k_c$ near the boundary.}

In order to identify the asymptotics of $k_c$ for $T^*$ near the edges of the interval $(0,1)$, we first compute the limit of $\that$ as $k\rightarrow\infty$. In the following, we suppress the dependence of $\that$ on $k$. By Taylor expansion,
\begin{equation}
\begin{split}
l_\theta(u_1^*) &\leq l_\theta\left(\tfrac{1}{2}\right) + \left(u_1^*-\tfrac{1}{2}\right)l_\theta'\left(\tfrac{1}{2}\right)\\
&\leq \theta\left(\tfrac{1}{2}\right)^k + \tfrac{1}{2}\log2 + \theta k\left(\tfrac{1}{2}\right)^k
= \theta\left(\tfrac{1}{2}\right)^k(1+k) + \tfrac{1}{2}\log 2,
\end{split}
\end{equation}
and $l_\theta(1) = \theta < l_\theta(u_2^*)$. This implies that
\begin{equation}
\hat\theta < \frac{\log2}{2[1-(\tfrac{1}{2})^k(1+k)]}.
\end{equation}
Also, $u_2^*(\theta,k)\in(\tfrac{k-1}{k},1)$ by \cite[Proposition 3.2]{RY13}. Hence
\begin{equation}
\begin{split}
l_{\theta,k}(u_2^*(\theta,k))
&\leq \theta - \tfrac{k-1}{2k}\log(\tfrac{k-1}{k}) 
- \tfrac{1}{2}(1-\tfrac{k-1}{k})\log(1-\tfrac{k-1}{k})\\
&= \theta - \tfrac{1}{2}\log(1-\tfrac{1}{k}) - \tfrac{1}{2k}\log(\tfrac{1}{k-1}),
\end{split}
\end{equation}
and $l_{\theta,k}(\tfrac{1}{2}) = \theta(\tfrac{1}{2})^k + \tfrac{1}{2}\log2 < l_{\theta,k}(u_1^*(\theta,k))$. This implies that
\begin{equation}
\begin{split}
\hat\theta > \frac{\log2 + \log(1-\frac{1}{k}) + \frac{1}{k}\log(\frac{1}{k-1})}{2[1-(\tfrac{1}{2})^k]}.
\end{split}
\end{equation}
Combining the bounds above, we obtain that $\that\to\frac{1}{2}\log2$ as $k\to\infty$.


\paragraph{$\bullet$ Scaling for $T^* \downarrow 0$.}

Let $y\in(\frac{1}{2},1)$. Then 
\begin{equation}
\label{eq:y}
\begin{split}
l_{\that,k}'(\tfrac{1}{2}+y^k)
&= \that k(\tfrac{1}{2}+y^k)^{k-1} - \tfrac{1}{2}\log\left(\tfrac{1+2y^k}{1-2y^k}\right)\\
&= \that k(\tfrac{1}{2}+y^k)^{k-1} - \tfrac{1}{2}\log\left(1+\tfrac{4y^k}{1-2y^k}\right)\\
&\leq \frac{\log2}{2[1-(\tfrac{1}{2})^k(1+k)]} k \left(\tfrac{1}{2}\right)^{k-1} -2y^k 
+ o(k\left(\tfrac{1}{2}\right)^k) + o(y^k) < 0
\end{split}
\end{equation}
as $k\rightarrow\infty$. Thus, $u_1^*(\that,k)<\frac{1}{2}+y^k$ for all $y\in(\frac{1}{2},1)$ and $k$ large enough. Hence $(\frac{1}{2}+y^{k_c})^{k_c}\geq T^*$ for $T^*$ small enough. We also have $T^*\geq(\frac{1}{2})^k$ for all $k$. Since this holds for all $y\in(\frac{1}{2},1)$ and $(\frac{1}{2}+y^{k})^{k}\sim(\frac{1}{2})^k$, we have $T^*\sim(\frac{1}{2})^{k_c}$.


\paragraph{$\bullet$ Scaling for $T^* \uparrow 1$.}
Let $x\in(0,1)$. Then 
\begin{equation}
\begin{split}
l'_{\that,k}(1-x^k)
&= k(\that(1-x^k)^{k-1}+\tfrac{1}{2}\log x)-\tfrac{1}{2}\log(1-x^k).\\        
\end{split}
\end{equation}
As $k\rightarrow\infty$, $(1-x^k)^{k-1}\rightarrow1$ and $\log(1-x^k)\rightarrow0$. Hence, if $-\frac{1}{2}\log x\geq\that$, then $l'_{\that,k}(1-x^k)<0$ for $k$ large enough, which implies that $u_2^*(\that,k)<1-x^k$. If $-\frac{1}{2}\log x<\that$, then $l'_{\that,k}(1-x^k)>0$, which implies that $u_2^*(\that,k)>1-x^k$. Recall that $\that\rightarrow\frac{1}{2}\log2$. Thus, choosing $x=\frac{1}{2}$, we get $(1-(\frac{1}{2})^{k_c})^{k_c}\sim T^*$, and so $k_c(\frac{1}{2})^{k_c} \sim1-T^*$.


\section{Proof of Theorem~\ref{thm:entropy}}
\label{s.proofentropy}
If $d=k$, then the statement of the theorem is vacuous, so we may assume that $d<k$. Let $T^*$ denote either $T_1^*(k)$ or $T_2^*(k)$. In this proof, we will often use that fact that $I(f)=I_{\frac{1}{2}}(f)-\frac{1}{2}\log2$. Any reference to the theory of replica symmetry is made with the implicit assumption that $p=\frac{1}{2}$.

Since there is ensemble equivalence for $T^*$, $(T^*,I((T^*)^{1/k}))$ lies on the convex minorant of $x\mapsto I(x^{1/k})$, and so $T^*\not\in(q_1(k)^k,q_2(k)^k)$, where $q_1(k),q_2(k)$ are defined as in the proof of Lemma \ref{lemma:convexminorantEE}. By \cite[Lemma A.5]{LZ16}, $q_1(k)<q_1(d)<q_2(d)<q_2(k)$, because $d<k$, with $d$ the largest degree of $H$. Hence, for all $T\in(T_1^*(k),q_1(d)]$ and $T\in[q_2(d),T_2^*(k))$, $(T,I(T^{1/d}))$ lies on the convex minorant of $x\mapsto I(x^{1/d})$, but $T$ is not in the region of ensemble equivalence. Thus, by \cite[Lemma 3.3]{LZ16}, $T$ is in the region of replica symmetry for $t(H,\cdot)$. This implies that $h\equiv T^{1/k}$ is the unique minimiser of
\begin{equation}
\begin{split}
\inf \{I(\h)\colon\,\h\in\Wt,\,t(H,\h)\geq T\} = \inf\{I(\h)\colon\,\h\in\Wt,\,t(H,\h)=T\}=\inf_{\h\in\Wt^*}I(\h).
\end{split}
\end{equation}
Furthermore, since $T$ is in the BEE-phase, we have $\theta^*=\that$. We conclude that
\begin{equation}
\begin{split}
s_\infty 
&= \sup_{\tilde{h}\in \Wt} \left[\theta^*T(\tilde{h})-I(\tilde{h})\right]
-\sup_{\tilde{h}\in \Wt^*} \left[\theta^*T(\tilde{h}) - I(\tilde{h})\right]\\
&= [\that T^*-I(T^{*\,1/k})] - [\that T-I(T^{1/k})]\\
&= \that (T^*-T) + [I(T^{1/k})-I(T^{*\,1/k})]\\
&= [J_k'(T^*)-\that](T-T^*)+J_k''(T^*)(T-T^*)^2+O((T-T^*)^3)\\
&= \frac{T^{*\,1/k-2}}{2k}\left\{\frac{1}{k}\left(1+\frac{T^{*\,1/k}}{1-T^{*\,1/k}}\right)+\left(\frac{1}{k}-1\right)\log\left(\frac{T^{*\,1/k}}{1-T^{*\,1/k}}\right)\right\}(T-T^*)^2+O((T-T^*)^3)
\end{split}
\end{equation}
as $T\to T^*$. The last equality follows from the fact that $J'_k(T^*)=\that$ (see the proof of Lemma \ref{lemma:convexminorantEE}).


\section{Proof of Theorem~\ref{thm:gap}}
\label{s.proofgap}

We first show that a graph sampled from the canonical ensemble converges to a probability distribution on a finite set of constant graphons. In \cite[Theorem 3.2 and Theorem 4.2]{CD13} this is shown for the exponential random graph model with a fixed parameter $\theta^*$. We adapt the proof to the case where we have a sequence of parameters $(\theta_n^*)_{n\in\N}$ converging to some $\theta^*$.

\begin{lemma}
\label{lemma:convergencecanonicalminimisers}
Let $G_n$ be a random graph drawn from the canonical ensemble $\Pcan$ with parameter $\theta^*_n$. Let $U(\theta)$ be the set of maximisers of \eqref{eq:variationalformulachatterjee} for some parameter $\theta$. Then (recall \eqref{graphondef})
\begin{equation}
\min_{u\in U(\theta^*_{\infty})}\des(\widetilde{h}^{G_n},\widetilde{u})\rightarrow0,\qquad n\rightarrow\infty
\end{equation}
in probability.
\end{lemma}

\begin{proof}
Let $\eta>0$ and define
\begin{equation}
\widetilde{A}(\theta,\eta):=\{\h\in\Wt\mid\des(\h,\widetilde{U}(\theta))\geq\eta\}.
\end{equation}
Recall from the proof of Theorem \ref{thm:subgraph} that $U(\theta)$ consists of a single point for $\theta\neq\that$ and two points for $\theta=\that$. Also recall the definition of the function $l_{\theta}$ from the proof of Theorem \ref{thm:subgraph}. We first prove the case that $\theta^*_{\infty}\neq\that$. Then $l_{\theta^*_{n}}$ converges to $l_{\theta^*_{\infty}}$ uniformly as $n\rightarrow\infty$, so $U(\theta_n^*)$ converges to $U(\theta_{\infty}^*)$. Here we assume without loss of generality that $\theta_n^*\neq\that$ and let $U(\theta)$ denote the single maximiser of $l_{\theta}$ by slight abuse of notation. Hence,
\begin{equation}\label{eq:convergenceminimisers}
\widetilde{A}(\theta_n^*,\eta)\subset\widetilde{A}(\theta_{\infty}^*,\tfrac{\eta}{2})
\end{equation}
for all $n$ large enough by the triangle inequality. We now adapt the arguments from the proof of \cite[Theorem 3.2]{CD13}.

By compactness of $\Wt$ and $\widetilde{U}(\theta)$, and upper semi-continuity of $\theta_\infty^*T-I$, it follows that
\begin{equation}
2\varepsilon:=\sup_{\h\in\Wt}\left[\theta_\infty^*T(\h)-I(\h)\right]
-\sup_{\h\in\widetilde{A}(\theta_\infty^*,\frac{\eta}{2})}\left[\theta_\infty^*T(\h)-I(\h)\right]>0.
\end{equation}
Since the $\theta_n^*T$ are all bounded functions and the sequence $(\theta_n^*)_{n\in\N}$ is bounded, there exists a finite set $R$ such that the intervals $\{(a,a+\varepsilon)\mid a\in R\}$ cover the range of $\theta^*_n T$ and $\theta_\infty^*T$ for all $n$ large enough. For each $a\in R$, let $\widetilde{F}^a(\theta_n^*):=(\theta_n^*T)^{-1}([a,a+\varepsilon])$. Now define
$\widetilde{A}^a(\theta_n^*,\eta):=\widetilde{A}(\theta_n^*,\eta)\cap\widetilde{F}^a(\theta_n^*)$ and $\widetilde{A}^a_n(\theta_n^*,\eta)=\widetilde{A}^a(\theta_n^*,\eta)\cap\widetilde{\cG}_n$. Choose $\delta=\frac{1}{2}\varepsilon$. Since $\theta_n^*\rightarrow\theta_\infty^*$, we have that
\begin{equation}\label{eq:convergencepreimage}
(\theta_n^*T)^{-1}([a,a+\varepsilon])\subset(\theta_\infty^*T)^{-1}([a-\delta,a+\varepsilon+\delta])=:\widetilde{G}^a
\end{equation}
for all $n$ large enough. Now define $\widetilde{B}^a:=\widetilde{A}(\theta_{\infty}^*,\frac{\eta}{2})\cap \widetilde{G}^a$ and $\widetilde{B}^a_n:=\widetilde{B}^a\cap\widetilde{\mathcal{G}}_n$. 

Using equations \eqref{eq:convergenceminimisers} and \eqref{eq:convergencepreimage}, we obtain $\widetilde{A}_n^a(\theta_n^*,\eta)\subset\widetilde{B}^a_n$. Hence,
\begin{equation}
\begin{split}
\Pcan(G_n\in\widetilde{A}(\theta_{n}^*,\eta)) 
&\leq \eee^{-n^2\psi_n(\theta^*_n)}\sum_{a\in R}e^{n^2(a+\varepsilon)}|\widetilde{A}_n^a(\theta_n^*,\eta)|\\
&\leq \eee^{-n^2\psi_n(\theta^*_n)}\sum_{a\in R}e^{n^2(a+\varepsilon)}|\widetilde{B}_n^a|\\
&\leq \eee^{-n^2\psi_n(\theta_n^*)}|R|\sup_{a\in R}\eee^{n^2(a+\varepsilon)}|\widetilde{B}_n^a|.
 \end{split}
\end{equation}
Using the large deviation principle for the Erd\H{o}s-R\'enyi random graph in \cite[Equation (8.1)]{CD13}, we obtain
\begin{equation}
\begin{split}
\limsup_{n\rightarrow\infty}\frac{\log|\widetilde{B}^a_n|}{n^2}\leq-\inf_{\h\in\widetilde{B}^a}I(\h).
\end{split}
\end{equation}
Also, by \cite[Lemma A.1]{dHMRS18}, we have
\begin{equation}
\lim_{n\to\infty} \psi_n(\theta_n^*) = \psi_{\infty}(\theta_\infty^*) = \sup_{\h\in\Wt}\left[\theta_\infty^*T(h)-I(h)\right].
\end{equation}
Combining these two results, we conclude
\begin{equation}\label{eq:canonicalconvergence}
\limsup_{n\to\infty}\frac{\log\Pcan(G_n\in\widetilde{A}(\theta_n^*,\eta))}{n^2}
\leq \sup_{a\in R}\left[a+\varepsilon-\inf_{\h\in\widetilde{B}^a}I(\h)\right]
-\sup_{\h\in\Wt}\left[\theta_\infty^*T(h)-I(h)\right].
\end{equation}

The remainder of the proof now follows exactly as in \cite{CD13}. Indeed, each $\widetilde{h}\in\widetilde{B}^a$, we have $\theta^*_{\infty}T(\h)\geq a-\delta$. Hence,
\begin{equation}
\sup_{\h\in\widetilde{B}^a}[\theta_{\infty}^*T(\h)-I(\h)]\geq a-\delta-\inf_{\h\in\widetilde{B}^a}I(\h).
\end{equation}
Substituting this into \eqref{eq:canonicalconvergence}, we get
\begin{equation}
\begin{split}
\limsup_{n\to\infty}\frac{\log\Pcan(G_n\in\widetilde{A}(\theta_n^*,\eta))}{n^2}\leq&\varepsilon+\delta+\sup_{a\in R}\sup_{\h\in\widetilde{B}^a}\left[\theta_{\infty}^*T(\h)-I(\h)\right]-\sup_{\h\in\Wt}\left[\theta_\infty^*T(\h)-I(\h)\right]\\
\leq&\varepsilon+\delta+\sup_{\h\in\widetilde{A}(\theta_\infty^*,\frac{\eta}{2})}\left[\theta_\infty^*T(\h)-I(\h)\right]-\sup_{\h\in\Wt}\left[\theta_\infty^*T(\h)-I(\h)\right]\\
\leq&\varepsilon+\delta-2\varepsilon=-\frac{\varepsilon}{2}.
\end{split}
\end{equation}
We have thus shown that
\begin{equation}
\des(\widetilde{h}^{G_n},\widetilde{U}(\theta^*_n))\rightarrow0,n\rightarrow\infty
\end{equation}
in probability. Since $U(\theta_n^*)\rightarrow U(\theta_{\infty}^*)$, this concludes the proof in the case $\theta_{\infty}\neq\that$.

Now assume that $\theta_{\infty}=\that$. Then \eqref{eq:convergenceminimisers} may no longer hold, since $U(\theta_{\infty}^*)$ now consists of two points, whereas $U(\theta_{n}^*)$ may consist of only one point. However, if we define $U'(\theta)$ as the set consisting of the two local maxima of $l_\theta$, and define $\widetilde{A}'(\theta,\eta)$ analogously, then the analogue of \eqref{eq:convergenceminimisers} does hold for all $n$ large enough. Here we use that the two local maxima of $l_{\theta^*_n}$ converge to the two local maxima of $l_{\theta_{\infty}^*}$. The rest of the proof then goes through as before to show that
\begin{equation}
\des(\widetilde{h}^{G_n},\widetilde{U}'(\theta^*_n))\rightarrow0
\end{equation}
in probability. Again using convergence of the local maxima, we obtain
\begin{equation}
\des(\widetilde{h}^{G_n},\widetilde{U}'(\theta^*_{\infty}))\rightarrow0
\end{equation}
in probability. However, for $\theta_{\infty}^*=\that$, we have that $\widetilde{U}'(\theta_{\infty}^*)=\widetilde{U}(\theta_{\infty}^*)$, which concludes the proof.
\end{proof}

\begin{corollary}
\label{cor:convergencecanonical}
Assume that $T^*$ is in the BEE-phase. Let $G_n$ be a random graph drawn from the canonical ensemble $\Pcan$. Then $h^{G_n}$ converges weakly to 
\begin{equation}
\frac{u_2^k-T^*}{u_2^k-u_1^k}\delta_{u_1}+\frac{T^*-u_1^k}{u_2^k-u_1^k}\delta_{u_2},
\end{equation} 
with $u_1<u_2$ the two maximisers of \eqref{eq:variationalformulachatterjee} for $\theta=\that$.
\end{corollary}

\begin{proof}
From Lemma~\ref{lemma:convergencecanonicalminimisers} it is clear that the laws of $(G_n)_{n\in\N}$ form a tight sequence of probability measures. Hence, by Prokhorov's Theorem, for every subsequence $(n_k)_{k\in\N}$ there exists a further subsequence $(n_{k_l})_{l\in\N}$ such that $(G_{n_{k_l}})_{l\in\N}$ converges weakly to the random graphon $p\delta_{u_1}+(1-p)\delta_{u_2}$ for some $p\in[0,1]$. Since the homomorphism density is continuous and bounded, this implies that $(\Ecan[t(H,G_{n_{k_l}})])_{l\in\N}$ converges to $pu_1^k+(1-p)u_2^k$. However, by the definition of the canonical ensemble, this sequence also converges to $T^*$. Hence
\begin{equation}
\begin{split}
T^*=\lim_{l\rightarrow\infty} \Ecan[t(H,G_{n_{k_l}})]=p u_1^k+(1-p)u_2^k.
\end{split}
\end{equation}
Solving for $p$, we obtain that $(G_{n_{k_l}})_{l\in\N}$ converges weakly to 
\begin{equation}
\frac{u_2^k-T^*}{u_2^k-u_1^k}\delta_{u_1}+\frac{T^*-u_1^k}{u_2^k-u_1^k}\delta_{u_2}.
\end{equation} 
Since the subsequence $(n_k)_{k\in\N}$ is arbitrary and the expression above does not depend on the chosen subsequence, we conclude that weak convergence holds for the sequence $(G_n)_{n\in\N}$.
\end{proof}

We can also show convergence of the microcanonical ensemble.

\begin{lemma}
\label{lemma:convergencemicrocanonical}
Let $G_n$ be a random graph drawn from the microcanonical ensemble $\Pmic$. Then $\widetilde{h}^{G_n}$ converges in probability to $\widetilde{F}^*$, with $\widetilde{F}^*$ the set of minimisers in $\Wt^*$ of $I$.
\end{lemma}

\begin{proof}
The proof is similar to the proof of \cite[Theorem 3.1]{CV11}. Fix $\varepsilon>0$ and let 
\begin{equation}
\widetilde{F}^\varepsilon:=\{\widetilde{h}\in\Wt^*\mid\des(\h,\widetilde{F}^*)>\varepsilon\}
\end{equation}
and
\begin{equation}
\widetilde{F}^\varepsilon_n:=\{\widetilde{h}\in\widetilde{F}_\varepsilon\mid\des(\h,\widetilde{F}^*)>\varepsilon,\,
\h=\widetilde{G}\text{ for some }G\in \mathcal{G}_n\}.
\end{equation}
Then, by \cite[(3.22) and Corollary 2.9]{dHMRS18},
\begin{equation}
\begin{split}
\lim_{n\to\infty} \frac{1}{n^2} \log\Pmic(\widetilde{F}^\varepsilon)
=& \lim_{n\to\infty} \frac{1}{n^2}\log(|\widetilde{F}^\varepsilon_n|\Pmic(G_n=G_n^*))\\
=&\inf_{\h\in\Wt^*}I(\h)+\limn\frac{1}{n^2}\log|\widetilde{F}_n^\varepsilon|\\
=&\inf_{\h\in\Wt^*}I(\h)-\inf_{\h\in\widetilde{F}^\varepsilon}I(\h),
\end{split}
\end{equation}
where $G_n^*$ is any graph in $\mathcal{G}_n$ such that $\widetilde{G}^*_n\in\Wt^*$. Since $\Wt^*$ is a compact set and $\widetilde{F}^\varepsilon$ does not contain any minimisers of $\inf_{\h\in\Wt^*}I(\h)$, we conclude that the expression above is negative, which implies that
\begin{equation}
\lim_{n \rightarrow \infty} \Pmic(\widetilde{F}^\varepsilon) = 0.
\end{equation}
\end{proof}

We next turn our attention to the largest eigenvalue. For a graph $G_n$ on $n$ vertices, $n^{-1}\lambda_{n}(G_n)$ equals the operator norm $\|h^{G_n}\|_{\op}$ of the empirical graphon of $G_n$. The operator norm is continuous and bounded, so we have 
\begin{equation}
\lim_{n\to\infty} n^{-1} \Ecan[\lambda_n] = pu_1+(1-p)u_2
= \frac{T^*(u_2-u_1)+u_1u_2(u_2^{k-1}-u_1^{k-1})}{u_2^k-u_1^k}=:f(T^*).
\end{equation}
If $T^*$ is in the region of replica symmetry for the subgraph $H$, then $h\equiv(T^*)^{1/k}$ is the unique minimiser of $I$ in $\Wt^*$. So, in this case,
\begin{equation}
\lim_{n\to\infty} n^{-1} \Emic[\lambda_n] = (T^*)^{1/k}>f(T^*),
\end{equation}
since the function $x \mapsto x^{1/k}$ is concave, $f$ is affine in $T^*$, and we have $f(u_1^k)=u_1=(u_1^k)^{1/k}$ and $f(u_2^k)=u_2=(u_2^k)^{1/k}$. 

The second part of the theorem follows from a simple Taylor expansion.



\end{document}